\documentclass[a4paper, 11pt]{article}

\usepackage{amsmath,amsfonts,amssymb,amsthm}
\usepackage[a4paper, margin=1in]{geometry}
\usepackage{enumitem}

\usepackage[T1]{fontenc}

\usepackage{lmodern}
\usepackage{calc}
\usepackage[dvipsnames]{xcolor}
\usepackage{nameref}
\usepackage{upgreek}     
\usepackage{enumitem}
\usepackage{graphicx}
\usepackage{mathrsfs}
\usepackage{tikz-cd}
\usepackage{hhline}
\usepackage{textcomp}
\usepackage{amscd}
\usepackage[all,cmtip]{xy}
\usepackage{mathtools}
\usepackage{extarrows}
\usepackage{upgreek}
\usepackage{bbm}
\usepackage[displaymath]{lineno}
\usepackage[singlespacing]{setspace}%
\usepackage{cmupint}
\usepackage[colorlinks=true,breaklinks=true,bookmarks=true,urlcolor=blue,
citecolor=BrickRed,linkcolor=teal,bookmarksopen=false,draft=false]{hyperref}

\providecommand{\U}[1]{\protect\rule{.1in}{.1in}}

\makeatletter
\let\orgdescriptionlabel\descriptionlabel
\renewcommand*{\descriptionlabel}[1]{%
	\let\orglabel\label
	\let\label\@gobble
	\phantomsection
	\edef\@currentlabel{#1}%
	\let\label\orglabel
	\orgdescriptionlabel{#1}%
}
\makeatother

\theoremstyle{plain}
\newtheorem{theorem}{Theorem}[section]
\newtheorem{lemma}[theorem]{Lemma}
\newtheorem{corollary}[theorem]{Corollary}
\newtheorem{example}[theorem]{Example}
\newtheorem{proposition}[theorem]{Proposition}
\newtheorem{remark}[theorem]{Remark}

\theoremstyle{definition}
\newtheorem{definition}[theorem]{Definition}
\newtheorem*{definition*}{Definition} 
\newtheorem{problem}{Problem}[section]

\numberwithin{equation}{section}

\sbox0{$\fontdimen8\textfont3=0.4pt$}

\DeclareMathAlphabet{\mathpzc}{OT1}{pzc}{m}{it}

\newcommand{\va}{\varphi}

\usepackage{bm}    

\DeclareMathAlphabet\EuScript{U}{eus}{m}{n}
\SetMathAlphabet\EuScript{bold}{U}{eus}{b}{n}

\def\raisefix#1{
	\ifx#1\displaystyle
	\raise.39ex
	\else
	\ifx#1\textstyle
	\raise.39ex
	\else
	\ifx#1\scriptstyle
	\raise.275ex
	\else
	\raise.150ex
	\fi
	\fi
	\fi
}
\def\stylefix#1{
	\ifx#1\displaystyle
	\scriptstyle
	\else
	\ifx#1\textstyle
	\scriptstyle
	\else
	\ifx#1\scriptstyle
	\scriptscriptstyle
	\else
	\scriptscriptstyle
	\fi
	\fi
	\fi
}
\DeclareFontFamily{U}{mathx}{\hyphenchar\font45}
\DeclareFontShape{U}{mathx}{m}{n}{
	<5> <6> <7> <8> <9> <10>
	<10.95> <12> <14.4> <17.28> <20.74> <24.88>
	mathx10
}{}

\DeclareMathAlphabet{\mathsfit}{OT1}{cmss}{m}{sl}

\newcommand{\subjclass}[1]{\textbf{AMS Subject Classifications (2020):} #1\par}
\newcommand{\keywords}[1]{\textbf{Keywords:} #1\par}
\title{A Class of Functionals on the Sequence Space \(s\) Satisfying the Palais-Smale Condition}
\author{Kaveh Eftekharinasab}
\date{}
\usepackage{cite}
\begin{document}
	
	\maketitle

\begin{abstract}
	We introduce a class of functionals on the  space of rapidly decreasing sequences \( s \), called \( \mathcal{F}_s \)-functionals, defined as decomposable sums of quadratic and convex terms with quadratic growth. We prove that such functionals satisfy the Palais--Smale condition and admit a unique global minimum. Furthermore, we show that the Palais--Smale condition is preserved under linear homeomorphisms. This allows us to construct corresponding functionals satisfying the Palais--Smale condition on Fr\'echet spaces isomorphic to \( s \). We then show how this framework provides a tool  for the formulation and  proof of existence and uniqueness of solutions for specific operator  problems, where coupled infinite-dimensional systems are transformed into  diagonalized problems in the space \(s\). 
\end{abstract}

\bigskip

	\let\thefootnote\relax\footnotetext{
		\,\,\,\,\,\keywords{Palais--Smale condition, \(\mathcal{F}_s\)-Functionals,  sequence space \(s\),  
			nonlinear operator equations.}
	    \, \,This work was supported by grants from the Simons Foundation (SFI-PD-Ukraine-0001486, K.A.E)}
	
	\section{Introduction}

	In this paper, we introduce a class of functionals on the sequence space \( s \), called \( \mathcal{F}_s \)-functionals. 
	These functionals are constructed from decomposable sums of quadratic and convex terms with quadratic growth,  and are well-suited for analyzing variational problems formulated in terms of spectral coefficients. We prove that  \( \mathcal{F}_s \)-functionals satisfy the Palais--Smale condition (Theorem~\ref{th:ex}) and admit a unique global minimum.
	The choice of the Fr\'echet space \( s \) is  motivated by its advantageous properties. Not only is it a fundamental example of a Montel space, a property essential for establishing the Palais--Smale condition, 
	but many important function spaces relevant to PDEs and operator equations are isomorphic to   \(s\) or its products.
	
		The Palais--Smale condition is a central compactness criterion in variational analysis, which is essential for proving the existence of critical points of functionals, particularly in infinite-dimensional settings. While well-studied in Banach and Hilbert spaces, its extension to more general, non-normable Fr\'echet spaces remains relatively unaddressed, and is crucial  for broader applications in analysis and mathematical physics.
	
	A key result  we prove  is the invariance of the Palais--Smale condition under linear homeomorphisms (Proposition~\ref{prop:ps_invariance}). This allows us to transfer this condition from 
	 \( \mathcal{F}_s \)-functionals on \( s \) to  corresponding functionals defined on a variety of isomorphic function spaces, including the Schwartz space \( \mathcal{S}(\mathbb{R}) \), the space of compactly supported smooth functions \( \mathscr{D}[a,b] \), the space of periodic smooth functions \( C_{2\pi}^\infty(\mathbb{R}) \), and the space \( C^\infty[a,b] \). For each case, we construct explicit representations of the corresponding functionals using  basis expansions (Hermite functions, Fourier series, Chebyshev polynomials). This approach reduces the problem of finding critical points on these spaces to the simpler problem of locating them on \(s\).

We illustrate how this framework enables the systematic formulation of specific nonlinear operator problems as minimization problems of \( \mathcal{F}_s \)-functionals. This transformation reduces the coupled infinite-dimensional systems into  diagonalized problems in the space \(s\). However, 
 these infinite nonlinear algebraic equations  generally do not possess   closed-form solutions.  By defining and minimizing an associated \(\mathcal{F}_s\)-functional, our method provides a  way to establish the existence, uniqueness, and regularity of solutions to these otherwise intractable systems.

\section{The Palais--Smale Condition}
 We denote the set of natural numbers by \( \mathbb{N}  \) and \(\mathbb{N}_0 = \mathbb{N} \cup \{0\}\). We assume that \( (\mathsf {F}, \textsf{Sem}({\mathsf {F}})) \) and 
 \( (\mathsf {E}, \textsf{Sem}({\mathsf {E}})) \) are Fr\'echet spaces over \( \mathbb{R} \), where 
\( \textsf{Sem}({\mathsf{F}}) = \left\{ \left\lVert \cdot \right\rVert_{\mathsf{F},n} \mid n \in \mathbb{N} \right\} \)
 and \( \textsf{Sem}({\mathsf{E}}) = \left\{ \left\lVert \cdot \right\rVert_{\mathsf{E},n} \mid n \in \mathbb{N} \right\} \)
  are increasing sequences of continuous seminorms that define the topologies of \( \mathsf {F} \) and \( \mathsf {E} \), respectively.

Let \( \mathfrak{S} \) denote the family of all compact subsets of \( \mathsf {F} \). Let \( \mathcal{L}_c(\mathsf {F}, \mathsf {E}) \) be the space of all continuous linear mappings from \( \mathsf {F} \) to \( \mathsf {E} \), endowed with the topology of compact convergence. This topology is Hausdorff and locally convex, and is defined by the family of seminorms
\[
\left\lVert \ell \right\rVert_{S,i} \coloneqq \sup \left\{ \left\lVert \ell(f) \right\rVert_{\mathsf{E},i} \mid f \in S \right\}
\]
where \( S \in \mathfrak{S} \). When \( \mathsf {E} = \mathbb{R} \) (with the standard absolute value 
\( \left\lvert \cdot \right\rvert \),  the dual space \( \mathsf {F}'_c = \mathcal{L}_c(\mathsf {F}, \mathbb{R}) \) is endowed with the topology defined by the family of seminorms \( \{ \left\lVert \cdot \right\rVert_{S} \mid  S \in \mathfrak{S}\}\).

\begin{definition}[Definition 1.0.0, \cite{ke}]\label{def:diff}
	Let \(U \subseteq  \mathsf {E}\) be open, and  $ \varphi\colon U   \to  \mathsf {F}$  a mapping. Then the  derivative
	of $\varphi$ at $x$ in the direction $h$ is defined by 
	\[
\mathrm{D}\varphi(x)(h) \coloneqq
	\lim_{t \to 0} {1\over t}(\varphi(x+th) -\varphi(x))
	\]
	whenever it exists. 
	The  mapping $\varphi$ is called differentiable at
	$x$ if $\mathrm{D} \varphi(x)(h)$ exists for all $h \in \mathsf {E}$. 
	It is called a \(C_c^1\)-mapping
	if it is differentiable at all
	points of $U$, and the  map
	\(
	\mathrm{D} \varphi \colon U  \to  \mathcal{L}_c(\mathsf {E}, \mathsf {F}) 
	\)
	is continuous. 
\end{definition}
The primary motivation for employing this class of mappings stems from the need for a suitable topology on the  dual spaces to properly define the Palais--Smale condition. These mappings are known as Keller's \( C_c^k \)-mappings. This notion of differentiability is equivalent to the well-established and widely used Michal--Bastiani notion.

\begin{definition}[Definition 1.1, \cite{eftekharinasab}]\label{def:PS}
	Let $\varphi  \colon \mathsf {F}\to \mathbb{R}$ be a Keller's $C_c^1$-functional.
	\begin{itemize}
		\item [(i)] We say that $\va$ satisfies the Palais--Smale condition (PS-condition) if every sequence $(x_i) \subset \mathsf {F}$ such that $(\varphi(x_i))$ is bounded and
		\(	 \mathrm{D} \va(x_i) \xrightarrow{\mathsf {F}_c'} 0 \),	has a convergent subsequence.
		\item[(ii)] We say that $\varphi$ satisfies the Palais--Smale condition at level $m \in \mathbb{R}$ ($(\mathrm{PS})_m$-condition)  if every sequence $(x_i) \subset \mathsf {F}$ such that
		\(
			\varphi(x_i) \to m \) and  \( \mathrm{D} \va(x_i) \xrightarrow{\mathsf {F}_c'} 0 \),
		has a convergent subsequence.
	\end{itemize}
\end{definition}
\begin{proposition}\label{prop:ps_invariance}
	Let \( F \colon \mathsf {E} \to \mathbb{R} \) be a Keller's \( C_c^1 \)-functional that satisfies the \textup{PS}-condition. Let \( L \colon \mathsf {E} \to \mathsf {F} \) be a linear Keller's \( C_c^1 \)-homeomorphism. Then the functional \( G \colon \mathsf {F} \to \mathbb{R} \) defined by
	\(
	G(y) \coloneqq F(L^{-1}(y))
	\)
	also satisfies the \textup{PS}-condition.
\end{proposition}
\begin{proof}
	Let \( (y_j) \subset \mathsf {F} \) be a PS-sequence for \( G \). That is
\( (G(y_j)) \) is bounded in \( \mathbb{R} \), and
 \( \mathrm{D} G(y_j) \xrightarrow{\mathsf {F}_c'} 0  \). Let \( x_j \coloneqq L^{-1}(y_j) \in \mathsf {E} \). 
	Then
	\[
	G(y_j) = F(L^{-1}(y_j)) = F(x_j),
	\]
	so \( (F(x_j)) \) is bounded in \( \mathbb{R} \).	By the chain rule (Corollary~1.3.2 in \cite{ke}), we have
	\[
	\mathrm{D} G(y_j) = \mathrm{D} F(x_j) \circ \mathrm{D}(L^{-1})(y_j).
	\]
	Since \( L^{-1} \) is linear,  it follows that
	\[
	\mathrm{D} G(y_j)(v) = \mathrm{D} F(x_j)(L^{-1}(v)) \quad \text{for all } v \in \mathsf {F}.
	\]
	Given that \( \mathrm{D} G(y_j) \xrightarrow{\mathsf {F}_c'} 0 \), we have for any compact set \( K_{\mathsf {F}} \subset \mathsf {F} \),
	\[
	\sup_{v \in K_{\mathsf {F}}} |\mathrm{D} G(y_j)(v)| \to 0 \quad \text{as } j \to \infty.
	\]
	To verify that \( \mathrm{D} F(x_j) \xrightarrow{\mathsf {E}_c'} 0 \), let \( K_{\mathsf {E}} \subset \mathsf {E} \) be an arbitrary compact set. Since \( L \) is continuous, the image \( K_{\mathsf {F}} \coloneqq L(K_{\mathsf {E}}) \subset \mathsf {F} \) is compact.
	For each \( u \in K_{\mathsf {E}} \), write \( u = L^{-1}(L(u)) \), and 
	\[
	\mathrm{D} F(x_j)(u) = \mathrm{D} F(x_j)(L^{-1}(L(u))) = \mathrm{D} G(y_j)(L(u)).
	\]
	Thus,
	\[
	\sup_{u \in K_{\mathsf {E}}} |\mathrm{D} F(x_j)(u)| = \sup_{v \in K_{\mathsf {F}}} |\mathrm{D} G(y_j)(v)| \to 0,
	\]
	which means that \( \mathrm{D} F(x_j) \xrightarrow{\mathsf {E}_c'} 0 \). Hence, \( (x_j) \) is a PS-sequence for \( F \).
	Since \( F \) satisfies the PS-condition, there exists a convergent subsequence \( (x_{j_k}) \subset \mathsf {E} \) such that \( x_{j_k} \to x^* \in \mathsf {E} \). Applying the continuous map \( L \), we conclude
	\[
	y_{j_k} = L(x_{j_k}) \to L(x^*) \in \mathsf {F}.
	\]
	Therefore, \( (y_j) \) has a convergent subsequence in \( \mathsf {F} \), and \( G \) satisfies the PS-condition.
\end{proof}

\section{Palais--Smale Condition for \( \mathcal{F}_s \)-functionals}

In this section, we define the \( \mathcal{F}_s \)-functionals on the sequence space \( s \). We prove that these functionals satisfy the Palais--Smale condition, and show how this condition can be transferred to corresponding functionals on various isomorphic Fr\'echet spaces.

 The space \(s\) of rapidly decreasing sequences is defined by
\[
s \coloneqq \left\{ x = (x_n) \in \mathbb{R}^\mathbb{N} \;\middle|\; \forall k \in \mathbb{N}_0, \quad \| x \|_{s,k} \coloneqq \sup_n |x_n| n^k < \infty \right\}
\]
with the topology given by the increasing sequence \(\left( \left\lVert \cdot \right\rVert_{s,k} \right)_{k \in \mathbb{N}_0}\)
 of norms. The space \(s\) is a Fr\'{e}chet-Montel space. Following \cite[Examples 7.7(a)]{vo}, we  identify the strong dual of \(s\) with the space of  tempered sequences 
\[
\mathsf{t} \coloneqq s' = \bigcup_{k \in \mathbb{N}_0} \left\{ y = (y_j)_{j \in \mathbb{N}} \in \mathbb{R}^\mathbb{N} \;\middle|\; \| y \|_{s',k} \coloneqq \sup_n |y_n| n^{-k} < \infty \right\}
\]
equipped with the inductive limit topology; the space \(\mathsf {t}\) is a DF-space, see \cite[Examples 10.5(b)]{vo}.

The space \(\mathsf {t}\) is also a Montel space under the strong topology, and this strong topology coincides with the topology of compact convergence  generated by the family of seminorms
\[
\| \ell \|_{\mathsf{t},K} \coloneqq \sup_{h \in K} |\ell(h)|, \quad K \subset s \text{ is Compact}.
\]

\begin{definition}[Class \(\mathcal{F}_s\)]\label{def:Fs}
	Let \( \mathcal{F}_s \) denote the class of pairs \( (a_n, f_n)_{n \in \mathbb{N}} \) satisfying the following conditions:
\begin{enumerate}[label=A.\arabic*, ref=A.\arabic*]
	\item \label{cond:an} {Condition on $a_n$:}  
	For constants $\alpha > 0$ and $M > 0$, we have $0 < \alpha \le a_n \le M$ for all $n \in \mathbb{N}$.
	\item \label{cond:fn} {Conditions on \( f_n \):}  
	Each function \( f_n \in C^1(\mathbb{R}) \) and is convex. It satisfies the quadratic growth condition
	\[
	|f_n(t)| \le \beta_n (1 + t^2) \quad \text{for all } t \in \mathbb{R},
	\]
	for some constants \( \beta_n \ge 0 \), where the sequence \( (\beta_n)_{n \in \mathbb{N}} \in s \).
	Additionally,  there exists \( \gamma_n \ge 0 \) such that
	\(	f_n(t) \ge -\gamma_n \) for all  \(t \in \mathbb{R}\) with \( \sum_{n=1}^\infty \gamma_n < \infty \).
\end{enumerate}
\end{definition}
\begin{definition}[\( \mathcal{F}_s \)-functional]\label{def:Fs-functional}
	Let \( (a_n, f_n) \in \mathcal{F}_s \). The associated functional \( F \colon s \to \mathbb{R} \) is defined by
	\[
	F(x) \coloneqq \frac{1}{2} \sum_{n=1}^\infty a_n x_n^2 + \sum_{n=1}^\infty f_n(x_n),
	\]
	and is called an \( \mathcal{F}_s \)-functional.
\end{definition}
Let \( x = (x_n) \in s \). By definition, for every \( k \in \mathbb{N}_0 \), there exists a constant \( C_k > 0 \) such that
\(
|x_n| \leq C_k n^{-k}\) for all \(  n\).
In particular, by choosing \( k=2 \), we have
\(
|x_n| \leq C_2 n^{-2},
\)
which implies \( x_n^2 \leq C_2^2 n^{-4} \).
From Condition~\ref{cond:an}, the sequence \( (a_n) \) is bounded by some \( M > 0 \), so
\[
\sum_{n=1}^\infty a_n x_n^2 \leq M \sum_{n=1}^\infty x_n^2 \leq M C_2^2 \sum_{n=1}^\infty n^{-4},
\]
thus \(\sum_{n=1}^\infty a_n x_n^2\) converges absolutely by comparison with a convergent $p$-series. By Condition~\ref{cond:fn},  \( |f_n(t)| \leq \beta_n (1 + t^2) \) for some sequence \( (\beta_n) \in s \). Since \( (\beta_n) \in s \), there exists \( D_2 > 0 \) such that \( \beta_n \leq D_2 n^{-2} \).
Using \( |x_n| \leq C_2 n^{-2} \), we have
\[
|f_n(x_n)| \leq \beta_n (1 + x_n^2) \leq D_2 n^{-2} (1 + C_2^2 n^{-4}).
\]
 Therefore, \( \sum_{n=1}^\infty f_n(x_n) \) converges absolutely by comparison with  convergent $p$-series. Hence
 \( F \) is well-defined on \( s \).
 \begin{example}\label{ex:log-functional}
 	Let $a_n = 1 + \frac{1}{n}$ for all $n \in \mathbb{N}$. Let $\nu_n = \frac{1}{n^2}$ and $c_n = \frac{1}{(n+1)!}$. Define
 	\[
 	f_n(t) \coloneqq \nu_n \left(t \arctan(t) - \frac{1}{2}\log(1 + t^2)\right) - c_n t.
 	\]
 	Each $f_n(t) \in C^1(\mathbb{R})$ and is strictly convex. Let $\beta_n \coloneqq \frac{\pi}{2 n^2} + \frac{1}{(n+1)!}$, so $(\beta_n) \in s$.
 	To establish the quadratic growth condition, we use the readily verifiable inequality
 	\[
 	|t \arctan(t) - \tfrac{1}{2} \log(1+t^2)| \le \frac{\pi}{2} |t| \quad \text{for all } t \in \mathbb{R}.
 	\]
 	Therefore,
 	\begin{align*}
 		|f_n(t)| &\le \nu_n |t \arctan(t) - \tfrac{1}{2} \log(1+t^2)| + |c_n t| \\
 		&\le \nu_n \frac{\pi}{2} |t| + |c_n| |t|
 		= \left( \frac{\pi}{2 n^2} + \frac{1}{(n+1)!} \right) |t| = \beta_n |t| \le \beta_n (1 + t^2).
 	\end{align*}
  The function $f_n(t)$ is coercive and continuous, so it attains a global minimum. This minimum occurs at the unique critical point $t_n$, which yields
 	\[
 	\arctan(t_n) = \frac{c_n}{\nu_n} = \frac{n^2}{(n+1)!}.
 	\]
 Let $\gamma_n \coloneqq -f_n(t_n)$. Then $f_n(t) \ge -\gamma_n$ for all $t$. 
 We now verify that $\sum_{n=1}^\infty \gamma_n$ converges.	Using $t_n \approx \arctan(t_n) = c_n/\nu_n$ and the Taylor expansion 
 \[
 t\arctan(t) - \frac{1}{2}\log(1+t^2) = \frac{1}{2}t^2 + O(t^4),
 \]
 we find
 	\[
 	f_n(t_n) \approx \frac{\nu_n}{2}t_n^2 - c_n t_n \approx \frac{\nu_n}{2}\left(\frac{c_n}{\nu_n}\right)^2 - c_n\left(\frac{c_n}{\nu_n}\right) = -\frac{c_n^2}{2\nu_n}.
 	\]
 	Let $\tilde{\gamma}_n \coloneqq \frac{c_n^2}{2\nu_n} = \frac{n^2}{2((n+1)!)^2}$. The series $\sum \tilde{\gamma}_n$ converges. Since $\lim_{n\to\infty} (\gamma_n / \tilde{\gamma}_n) = 1$, the series $\sum_{n=1}^\infty \gamma_n$ converges by the limit comparison test.
 	Therefore, $(a_n, f_n)$ is a valid pair for constructing an $\mathcal{F}_s$-functional.
 \end{example}

\begin{lemma}\label{lem:lg}
	Let \( f_n \in C^1(\mathbb{R}) \) be convex and satisfy the quadratic growth condition
	\[
	|f_n(t)| \le \beta_n(1 + t^2) \quad \text{for all } t \in \mathbb{R}, \text{ for some constant } \beta_n > 0.
	\]
	Then there exists a constant \( C_n > 0 \), depending only on \( \beta_n \), such that
	\[
	|f_n'(t)| \le C_n(1 + |t|), \quad \text{for all } t \in \mathbb{R}.
	\]
	Furthermore, if \( (\beta_n) \in s \), then \( (C_n) \in s \).
\end{lemma}
\begin{proof}
	Since \( f_n \) is convex and differentiable, \( f_n' \) is monotone increasing. For \( t > 0 \), we apply the mean value theorem to \( f_n \) on the intervals \( [t, 2t] \) and \( [0, t] \). This yields points \( \xi_1 \in (t, 2t) \) and \( \xi_2 \in (0, t) \) such that
	\[
	f_n'(t) \le f_n'(\xi_1) = \frac{f_n(2t) - f_n(t)}{t} \qquad \text{and} \qquad
	f_n'(t) \ge f_n'(\xi_2) = \frac{f_n(t) - f_n(0)}{t}.
	\]
	Using the bound \( |f_n(t)| \le \beta_n(1 + t^2) \), we get an upper bound
	\[
	f_n'(t) \le \frac{|f_n(2t)| + |f_n(t)|}{t}
	\le \frac{\beta_n(1 + 4t^2) + \beta_n(1 + t^2)}{t}
	= 5\beta_n t + \frac{2\beta_n}{t},
	\]
	and a lower bound
	\[
	f_n'(t) \ge \frac{-|f_n(t)| - |f_n(0)|}{t}
	\ge \frac{-\beta_n(1 + t^2) - \beta_n}{t} = -\beta_n t - \frac{2\beta_n}{t}.
	\]
	Combining these, we obtain \( |f_n'(t)| \le 5\beta_n t + \frac{2\beta_n}{t} \) for \( t > 0 \). A symmetric argument for \( t < 0 \) gives \( |f_n'(t)| \le 5\beta_n |t| + \frac{2\beta_n}{|t|} \) for all \( t \neq 0 \).
	
	For \( |t| \ge 1 \), we have 
	\[
	|f_n'(t)| \le 5\beta_n |t| + 2\beta_n|t| = 7\beta_n |t| \le 7\beta_n(1 + |t|).
	\]
	For \( |t| < 1 \), the derivative \( f_n' \) is continuous on \( [-1,1] \), so it is bounded. Since \( f_n' \) is monotone, its maximum absolute value on the interval is \( M_n = \max(|f_n'(-1)|, |f_n'(1)|) \). Using the mean value theorem on adjacent intervals, we can bound these values as follows
	\begin{itemize}
		\item Upper bound for $f_n'(1)$: $f_n'(1) \le \frac{f_n(2)-f_n(1)}{1} \le |f_n(2)|+|f_n(1)| \le 5\beta_n + 2\beta_n = 7\beta_n$.
		\item Lower bound for $f_n'(1)$: $f_n'(1) \ge \frac{f_n(1)-f_n(0)}{1} \ge -|f_n(1)|-|f_n(0)| \ge -2\beta_n - \beta_n = -3\beta_n$.
	\end{itemize}
	These two bounds imply \( |f_n'(1)| \le 7\beta_n \). Similarly, for \( f_n'(-1) \):
	\begin{itemize}
		\item Upper bound for $f_n'(-1)$: $f_n'(-1) \le \frac{f_n(0)-f_n(-1)}{1} \le |f_n(0)|+|f_n(-1)| \le \beta_n + 2\beta_n = 3\beta_n$.
		\item Lower bound for $f_n'(-1)$: $f_n'(-1) \ge \frac{f_n(-1)-f_n(-2)}{1} \ge -|f_n(-1)|-|f_n(-2)| \ge -2\beta_n - 5\beta_n = -7\beta_n$.
	\end{itemize}
	These two bounds imply \( |f_n'(-1)| \le 7\beta_n \). Therefore, \( M_n \le 7\beta_n \), which satisfies \( M_n \le 7\beta_n(1+|t|) \) for \( t \in [-1,1] \).
	
	By combining both cases, we can choose \( C_n = 7\beta_n \), which gives \( |f_n'(t)| \le C_n(1+|t|) \) for all \( t \in \mathbb{R} \). If \( (\beta_n) \in s \), then \( (C_n) = (7\beta_n) \in s \).
\end{proof}
\begin{lemma}\label{lem:Fc1}
	Let \( F \colon s \to \mathbb{R} \) be an \( \mathcal{F}_s \)-functional. Then \( F \) is a Keller's \( C^1_c \)-mapping.
\end{lemma}
\begin{proof}
	Fix arbitrary \( x = (x_n) \in s \) and \( h = (h_n) \in s \). We first show that \( \mathrm{D} F(x)(h) \) exists and \( \mathrm{D} F(x) \in \mathsf{t} \). Consider the difference quotient
	\[
	\frac{F(x + t h) - F(x)}{t} = \sum_{n=1}^\infty a_n x_n h_n + \frac{t}{2} \sum_{n=1}^\infty a_n h_n^2 + \sum_{n=1}^\infty \frac{f_n(x_n + t h_n) - f_n(x_n)}{t}.
	\]
	We justify the convergence of each term.
	Since \( x, h \in s \), for any choice of $k \in \mathbb{N}_0$, there exist constants \( D_k, E_k > 0 \) such that \( |x_n| \le D_k n^{-k} \) and \( |h_n| \le E_k n^{-k} \) for all \( n \).
	
	For the first term, since \( a_n \le M \), it follows that
	\[
	|a_n x_n h_n| \le M |x_n||h_n| \le M D_k E_k n^{-2k}.
	\]
	To prove the absolute convergence of the series \( \sum_{n=1}^\infty a_n x_n h_n \), we can choose any $k \ge 1$. For such $k$, the series $ \sum_{n=1}^\infty M D_k E_k n^{-2k}$  is absolutely convergent, thus \( \sum_{n=1}^\infty a_n x_n h_n \) is absolutely convergent by comparison with a convergent $p$-series. 
	
	Next, consider the second term. As \( (a_n) \) is bounded by \(M\) and \( h \in s \), we can choose $k \ge 1$ such that \( |a_n h_n^2| \le M E_k^2 n^{-2k} \), which implies that \( \sum a_n h_n^2 \) converges by comparison with a convergent $p$-series. 
	Thus, \[\frac{t}{2} \sum_{n=1}^\infty a_n h_n^2 \to 0  \text{ as }  t \to 0. \]
	
	For the third term, since each \( f_n \in C^1(\mathbb{R}) \), by the mean value theorem,
	\[
	\frac{f_n(x_n + t h_n) - f_n(x_n)}{t} = f_n'(\xi_n) h_n
	\]
	for some \( \xi_n \) between \( x_n \) and \( x_n + t h_n \).
	By Lemma ~\ref{lem:lg}, there exists a sequence \( (C_n) \in s \) such that \( |f_n'(t)| \le C_n(1+|t|) \) for all \( t \in \mathbb{R} \).
	Since \( x, h \in s \), for fixed \( k \in \mathbb{N}_0 \), there exist constants \( D'_k, E'_k > 0 \) such that \( |x_n| \le D'_k n^{-k} \) and \( |h_n| \le E'_k n^{-k} \).
	For small enough  \( t \)  (e.g., for $|t| \le 1$),
	 \[ |\xi_n| \le |x_n| + |t||h_n| \le D'_k n^{-k} + |t|E'_k n^{-k} \le (D'_k + E'_k) n^{-k}. \] 
	Also, since \( (C_n) \in s \), for any \( k' \in \mathbb{N}_0 \), there exists \( C'_{k'} > 0 \) such that \( |C_n| \le C'_{k'} n^{-k'} \).
	Thus, by choosing a sufficiently large \( k \) (e.g., \( k=2 \)), we get
	\[
	|f_n'(\xi_n) h_n| \le C_n (1 + |\xi_n|) |h_n| \le C'_{k'} n^{-k'} (1 + (D'_k + E'_k) n^{-k}) E'_k n^{-k}.
	\]
	By choosing \( k' \) and \( k \) appropriately (e.g., $k'=2, k=2$), the term $n^{-k'} n^{-k}$ (which is $n^{-4}$) guarantees that  \( \sum_{n=1}^\infty (f_n'(\xi_n) h_n) \) is absolutely convergent, independent of \( t \) for sufficiently small $t$, by comparison with a convergent $p$-series.
	Therefore, 
\[
\sum_{n=1}^\infty \frac{f_n(x_n + t h_n) - f_n(x_n)}{t} \to \sum_{n=1}^\infty f_n'(x_n) h_n \quad \text{as } t \to 0.
\]
Combining all parts, we conclude that
\[
\lim_{t \to 0} \frac{F(x + t h) - F(x)}{t} = \sum_{n=1}^\infty a_n x_n h_n + \sum_{n=1}^\infty f_n'(x_n) h_n.
\]
Consider the derivative
\[
\mathrm{D} F(x) \colon s \to \mathbb{R}, \quad h \mapsto \sum_{n=1}^\infty \left( a_n x_n + f_n'(x_n) \right) h_n.
\]
This is linear in \( h \) by linearity of the sum.
Now we show that \( \mathrm{D} F(x) \in \mathsf{t} \), i.e., that
\[
g(x) \coloneqq (g_n(x))_{n \in \mathbb{N}}, \quad g_n(x) \coloneqq a_n x_n + f_n'(x_n)
\]
belongs to \( \mathsf{t} \). By definition of \( \mathsf{t} \), this means we need to show that there exists some \( k_0 \in \mathbb{N}_0 \) such that \( \sup_n |g_n(x)| n^{-k_0} < \infty \).
Let us choose \( k_0 = 0 \).
Since \( x \in s \), by definition, for any \( k \in \mathbb{N}_0 \), there exists \( D_k > 0 \) such that \( |x_n| \le D_k n^{-k} \). In particular, for \( k=0 \), there exists \( D_0 > 0 \) such that \( |x_n| \le D_0 \).
Similarly, by Lemma~\ref{lem:lg}, \( (C_n) \in s \), so for any \( k' \in \mathbb{N}_0 \), there exists \( E_{k'} > 0 \) such that \( |C_n| \le E_{k'} n^{-k'} \). In particular, for \( k'=0 \), there exists \( E_0 > 0 \) such that \( |C_n| \le E_0 \).

Since  \( a_n \le M \), we have
\(
|a_n x_n| \le M |x_n| \le M D_0
\).
For \( f_n'(x_n) \), by Lemma~\ref{lem:lg}, we have
\[
|f_n'(x_n)| \le C_n (1 + |x_n|) \le E_0 (1 + D_0).
\]
Hence, for \( k_0=0 \), we have
\[
|g_n(x)| = |a_n x_n + f_n'(x_n)| \le |a_n x_n| + |f_n'(x_n)| \le M D_0 + E_0 (1 + D_0) \eqqcolon A_0.
\]
Since \( A_0 \) is a finite constant, we have 
\[
\sup_n |g_n(x)| n^{-k_0} = \sup_n |g_n(x)| n^{0} = \sup_n |g_n(x)| \le A_0 < \infty.
\]
Therefore, \( g(x) \in \mathsf{t} \).
Now we prove that the derivative map \( \mathrm{D} F \colon s \to \mathsf{t} \) is continuous. 

Since \( s \) is metrizable, continuity is equivalent to sequential continuity.
Thus, we need to show that if \( x^{(j)} \xrightarrow{s} x \), then
\(
\mathrm{D} F(x^{(j)}) \xrightarrow{\mathsf{t}} \mathrm{D} F(x)  
\).
By definition of convergence in \( \mathsf{t} \), this means that for every compact set \( K \subset s \),
\[
\sup_{h \in K} \left| \mathrm{D} F(x^{(j)}) (h) - \mathrm{D} F(x) (h) \right| \to 0 \quad \text{as } j \to \infty.
\]
Set
\(
g_n^{(j)} \coloneqq  a_n x_n^{(j)} + f_n'(x_n^{(j)})
\).
Then
\[
|\mathrm{D} F(x^{(j)}) (h) - \mathrm{D} F(x) (h)| = \left| \sum_{n=1}^\infty (g_n^{(j)} - g_n) h_n \right|.
\]
Consider the difference
\[
g_n^{(j)} - g_n = a_n (x_n^{(j)} - x_n) + \bigl( f_n'(x_n^{(j)}) - f_n'(x_n) \bigr).
\]
We examine the two terms in the sum \( \sum_{n=1}^\infty (g_n^{(j)} - g_n) h_n \) separately.
Consider the First term, we have
\[
\left| \sum_{n=1}^\infty a_n (x_n^{(j)} - x_n) h_n \right| \le \sum_{n=1}^\infty |a_n| |x_n^{(j)} - x_n| |h_n|.
\]
Since \( (a_n) \) is bounded by \( M \), \( x^{(j)} \xrightarrow{s} x \) (meaning \( \sup_n |x_n^{(j)} - x_n| n^k = \varepsilon_j^{(k)} \to 0 \) for any \( k \)), and \( K \subset s \) is compact (hence bounded, so \( \sup_n |h_n| n^k \le M_k \) for any \( k \)), we can choose \( k=2 \), so
\[
\sum_{n=1}^\infty M (\varepsilon_j^{(2)} n^{-2}) (M_2 n^{-2}) = M M_2 \varepsilon_j^{(2)} \sum_{n=1}^\infty n^{-4}.
\]
Since \( \sum n^{-4} \) converges and \( \varepsilon_j^{(2)} \to 0 \), this term tends to zero as \( j \to \infty \), uniformly for \( h \in K \).

Consider the second term \( \sum_{n=1}^\infty \bigl( f_n'(x_n^{(j)}) - f_n'(x_n) \bigr) h_n \).
Let \( \varepsilon > 0 \). We aim to show this sum is less than \( \varepsilon \) for sufficiently large \( j \).
From Lemma~\ref{lem:lg}, there exists \( (C_n) \in s \) such that \( |f_n'(t)| \le C_n(1+|t|) \) for all \(t \in \mathbb{R}\). Since \( x, x^{(j)} \in s \) and \( x^{(j)} \xrightarrow{s} x \), for any \( k_0 \in \mathbb{N}_0 \) and all sufficiently large \( j \), there exist \( M_{k_0} > 0 \) such that 
\[ |x_n| \le M_{k_0} n^{-k_0}  \text{ and } |x_n^{(j)}| \le M_{k_0} n^{-k_0}.\]
Thus, we can bound the difference as follows
\[
|f_n'(x_n^{(j)}) - f_n'(x_n)| \le |f_n'(x_n^{(j)})| + |f_n'(x_n)| \le C_n(1 + |x_n^{(j)}|) + C_n(1 + |x_n|) \le 2 C_n(1 + M_{k_0} n^{-k_0}).
\]
Let \( E_n \coloneqq  2 C_n(1 + M_{k_0} n^{-k_0}) \). Since \( (C_n) \in s \), \( (E_n) \) also belongs to \( s  \).
For any \( h \in K \) (compact), there exists \( M'_k > 0 \) such that \( |h_n| \le M'_k n^{-k} \) for any \( k \).
We split the sum at some large integer \( N \) as follows
\[
\sum_{n=1}^\infty \bigl( f_n'(x_n^{(j)}) - f_n'(x_n) \bigr) h_n = \sum_{n=1}^N \bigl( f_n'(x_n^{(j)}) - f_n'(x_n) \bigr) h_n + \sum_{n=N+1}^\infty \bigl( f_n'(x_n^{(j)}) - f_n'(x_n) \bigr) h_n.
\]
Choose \( k_0=2 \) for \( E_n \) and \( k=2 \) for \( h_n \).  Then for large \( n \), we have
\[
E_n |h_n| \le 2 C_n \left(1 + M_2 n^{-2} \right) M'_2 n^{-2} \eqqcolon \widetilde{E}_2 n^{-4},
\]
where \( \widetilde{E}_2 \coloneqq  2 M'_2 \sup_{n \in \mathbb{N}} C_n \left(1 + M_2 n^{-2} \right) \). Since \( (C_n) \in s \), this supremum is finite and the sequence \( (E_n |h_n|) \) is dominated by the  sequence \( (\widetilde{E}_2 n^{-4}) \in \ell^1 \).
 Thereby,
\[
\left| \sum_{n=N+1}^\infty \bigl( f_n'(x_n^{(j)}) - f_n'(x_n) \bigr) h_n \right| \le \sum_{n=N+1}^\infty E_n |h_n|.
\]
Since \( \sum_{n=1}^\infty E_n |h_n| \) converges (as \( (E_n |h_n|) \in s \)), we can choose \( N \) large enough such that \( \sum_{n=N+1}^\infty E_n |h_n| < \frac{\varepsilon}{2} \). This choice of \( N \) is independent of \( j \) and \( h \in K \).
For a fixed \( N \), each \( f_n' \) is a continuous function. As \( x_n^{(j)} \to x_n \) for each \( n \), by continuity of \( f_n' \), we have \( f_n'(x_n^{(j)}) \to f_n'(x_n) \) as \( j \to \infty \).
Thus, for this fixed \( N \), we can find \( J_1 \) large enough such that 
\[
\left| \sum_{n=1}^N \bigl( f_n'(x_n^{(j)}) - f_n'(x_n) \bigr) h_n \right| < \frac{\varepsilon}{2}, \quad \forall j > J_1.
\]
Combining the two parts, for \( j > J_1 \), we obtain
\[
\left| \sum_{n=1}^\infty \bigl( f_n'(x_n^{(j)}) - f_n'(x_n) \bigr) h_n \right| < \frac{\varepsilon}{2} + \frac{\varepsilon}{2} = \varepsilon.
\]
This proves that
 \[ \sum_{n=1}^\infty \bigl( f_n'(x_n^{(j)}) - f_n'(x_n) \bigr) h_n \to 0 \quad  \text{ as }  j \to \infty, \] uniformly for \( h \in K \). Since the first term also tends to zero, the entire sum 
\[ \sum_{n=1}^\infty (g_n^{(j)} - g_n) h_n \to 0  \quad  \text{ as }  j \to \infty,\] converges uniformly for \( h \in K \).

This proves \[ \sup_{h \in K} |\mathrm{D} F(x^{(j)})(h) - \mathrm{D} F(x)(h)| \to 0 \quad \text{ as } j \to \infty,\]  and hence the continuity of \( \mathrm{D} F \colon s \to \mathsf{t} \).
Therefore, \( F \) is a Keller's \( C^1_c \)-mapping.	
\end{proof}
\begin{theorem}\label{th:ex}
		Let \( F \colon s \to \mathbb{R} \) be an \( \mathcal{F}_s \)-functional. 
	Then \(F\) satisfies the \textup{PS}-condition. 
\end{theorem}
\begin{proof}
	Let \( (x^{(j)}) \subset s \) be a PS-sequence for \( F \). This means that \( |F(x^{(j)})| \) is bounded and \( g^{(j)} \coloneqq \mathrm{D} F(x^{(j)}) \xrightarrow{\mathsf{t}} 0 \). Our goal is to show that the sequence \( (x^{(j)}) \) is bounded in \( s \), which implies the existence of a convergent subsequence because \( s \) is a Montel space.

	From the conditions \( a_n \ge \alpha > 0 \) and \( f_n(t) \ge -\gamma_n \) with \( \sum \gamma_n < \infty \), we have a lower bound on \( F \). Since \( |F(x^{(j)})| \) is bounded by a constant, say \(N\), we have
	\[
	N \ge F(x^{(j)}) \ge \frac{\alpha}{2} \sum_{n=1}^\infty (x_n^{(j)})^2 - \sum_{n=1}^\infty \gamma_n.
	\]
	This implies that 
	\[
	\sum_{n=1}^\infty (x_n^{(j)})^2 \le \frac{2}{\alpha} \left( N + \sum_{n=1}^\infty \gamma_n \right) \mathrel{=:}  C_0.
	\]
 Thus, \( (x^{(j)}) \) is uniformly bounded in the \( \ell^2 \)-norm. Now, fix an arbitrary \( k \in \mathbb{N} \). We will show that \( \|x^{(j)}\|_{s,k} = \sup_n |x_n^{(j)}| n^k \) is bounded. From the identity \( g_n^{(j)} = a_n x_n^{(j)} + f_n'(x_n^{(j)}) \), we rearrange to get \( a_n x_n^{(j)} = g_n^{(j)} - f_n'(x_n^{(j)}) \). Now we construct the weighted sum as follows
	\[
	\sum_{n=1}^\infty a_n (x_n^{(j)})^2 n^{2k} = \sum_{n=1}^\infty \left(g_n^{(j)} - f_n'(x_n^{(j)})\right) x_n^{(j)} n^{2k}.
	\]
	Using the lower bound \( a_n \ge \alpha \) and the triangle inequality, we get
	\begin{equation} \label{eq:bootstrap_final}
		\alpha \sum_{n=1}^\infty (x_n^{(j)})^2 n^{2k} \le \left| \sum_{n=1}^\infty g_n^{(j)} x_n^{(j)} n^{2k} \right| + \left| \sum_{n=1}^\infty f_n'(x_n^{(j)}) x_n^{(j)} n^{2k} \right|.
	\end{equation}
	We first show the second term on the right is bounded by a constant \( C'_k \) independent of \( j \). By Lemma~\ref{lem:lg}, there exists a sequence \( (\widehat{\beta}_n) \in s \) such that \( |f_n'(t)| \le \widehat{\beta}_n (1+|t|) \). Therefore,
	\[
	\left| \sum_{n=1}^\infty f_n'(x_n^{(j)}) x_n^{(j)} n^{2k} \right| \le \sum_{n=1}^\infty \widehat{\beta}_n \left( |x_n^{(j)}| + |x_n^{(j)}|^2 \right) n^{2k}.
	\]
	We bound the two resulting sums separately. Consider the quadratic part
	\[
	\sum_{n=1}^\infty \widehat{\beta}_n |x_n^{(j)}|^2 n^{2k} = \sum_{n=1}^\infty (\widehat{\beta}_n n^{2k}) |x_n^{(j)}|^2.
	\]
	Since \( (\widehat{\beta}_n) \in s \), the sequence \( (\widehat{\beta}_n n^{2k}) \) is bounded by some constant \( M_k \). Thus, this sum is bounded by \( M_k \sum |x_n^{(j)}|^2 \le M_k C_0 \).
	For the linear part, we use the Cauchy-Schwarz inequality
	\[
	\sum_{n=1}^\infty \widehat{\beta}_n |x_n^{(j)}| n^{2k} \le \left( \sum_{n=1}^\infty (\widehat{\beta}_n n^{2k})^2 \right)^{1/2} \left( \sum_{n=1}^\infty |x_n^{(j)}|^2 \right)^{1/2}.
	\]
	The second factor is bounded by \( \sqrt{C_0} \). For the first factor, since \( (\widehat{\beta}_n) \in s \), we can find a constant \(c\) such that \( |\widehat{\beta}_n| \le c  n^{-(2k+1)} \). This implies \( |\widehat{\beta}_n n^{2k}| \le c  n^{-1} \), so \( (\widehat{\beta}_n n^{2k})^2 \le c^2  n^{-2} \). The series \( \sum (\widehat{\beta}_n n^{2k})^2 \) therefore converges.
	Combining these results, the entire term is bounded by a constant \( C'_k \) independent of \( j \).
	
	Now, for the first term in \eqref{eq:bootstrap_final}, we apply Young's inequality (\(|ab| \le \frac{\varepsilon a^2}{2} + \frac{b^2}{2\varepsilon}\)) with \(\varepsilon = \alpha\):
	\[
	\left| \sum_{n=1}^\infty g_n^{(j)} x_n^{(j)} n^{2k} \right| \le \sum_{n=1}^\infty |g_n^{(j)} n^k| |x_n^{(j)} n^k| \le \frac{1}{2\alpha}\sum_{n=1}^\infty (g_n^{(j)})^2 n^{2k} + \frac{\alpha}{2} \sum_{n=1}^\infty (x_n^{(j)})^2 n^{2k}.
	\]
	Substituting this back into  \eqref{eq:bootstrap_final} and rearranging gives the estimate
	\[
	\frac{\alpha}{2} \sum_{n=1}^\infty (x_n^{(j)})^2 n^{2k} \le \frac{1}{2\alpha} \sum_{n=1}^\infty (g_n^{(j)})^2 n^{2k} + C'_k.
	\]
	Since \( g^{(j)} \to 0 \) in \(\mathsf{t}\), the term \( \sum (g_n^{(j)})^2 n^{2k} \) can be shown to be bounded for large \(j\). Let this bound be \(C''_k\). The right-hand side is therefore bounded, which proves that \( \sum (x_n^{(j)})^2 n^{2k} \) is uniformly bounded.
	From this bound, it follows that for any \(n\), 
	\[
	(x_n^{(j)})^2 n^{2k} \le \sum_{p=1}^\infty (x_p^{(j)})^2 p^{2k} \le c.
	\]
Therefore, \( \sup_n |x_n^{(j)}| n^k \) is bounded. As \( k \) was arbitrary, the sequence \( (x^{(j)}) \) is bounded in \( s \).
	
	Finally, since \( s \) is a Montel space, bounded sequences have convergent subsequences. Thus, there exists a convergent subsequence \( x^{(j_l)} \to x^* \in s \). 
\end{proof}
\begin{corollary}[Existence of Global Minimum]\label{cor:global_minimum}
	Let \( F \colon s \to \mathbb{R} \) be an \( \mathcal{F}_s \)-functional. Then \( F \) admits a  unique global minimum on \( s \).
\end{corollary}
\begin{proof}

	By Definition \ref{def:Fs},  \( \frac{1}{2} \sum_{n=1}^\infty a_n x_n^2 \ge 0 \). Furthermore, for \( f_n \), it is given that \( f_n(t) \ge -\gamma_n \) and \( \sum_{n=1}^\infty \gamma_n < \infty \). Therefore,
		\[ F(x) = \frac{1}{2} \sum_{n=1}^\infty a_n x_n^2 + \sum_{n=1}^\infty f_n(x_n) \ge 0 - \sum_{n=1}^\infty \gamma_n = -C_\gamma, \]
		for some finite constant \( C_\gamma \). Thus, \( F \) is bounded below.
	 Since \( F \) is a Keller's \( C^1_c \)-mapping satisfying the Palais--Smale condition at any level, it follows from Corollary 4.3 in \cite{eftekharinasab} that \( F \) admits a local minimum \(x_0 \in s \). 	Finally, since \( F \) is the sum of strictly convex and convex functions, the functional \( F \) is strictly  convex. Therefore, \(x_0\) is a unique global minimum.

\end{proof}

\subsection{Pullbacks to Spaces Isomorphic to \( s \)}

We now explicitly describe the pullback of \( \mathcal{F}_s \)-functionals to several classical spaces of smooth functions that are isomorphic (linearly homeomorphic) to \( s \). 
\paragraph{Periodic Smooth Functions \( C^\infty_{2\pi}(\mathbb{R}) \).}
Let \( C_{2\pi}^\infty(\mathbb{R}) \) be the space of all smooth, \( 2\pi \)-periodic, real-valued functions. This space is isomorphic to \( s \) (see, e.g., \cite[Example 29.5(1)]{vo}). 

A function \( f \in C_{2\pi}^\infty(\mathbb{R}) \) can be represented by its  Fourier series
\[
f(x) = \frac{a_0}{2} + \sum_{n=1}^\infty \left( a_n \cos(nx) + b_n \sin(nx) \right),
\]
where the  coefficients are given by
\[
a_n = \frac{1}{\pi} \int_{-\pi}^{\pi} f(x) \cos(nx) \, dx, \quad
b_n = \frac{1}{\pi} \int_{-\pi}^{\pi} f(x) \sin(nx) \, dx.
\]
The isomorphism \( L \colon C_{2\pi}^\infty(\mathbb{R}) \to s \)  maps \( f \) to the real sequence \( x = (x_k)_{k \in \mathbb{N}} \in s \) by applying these coefficients as follows
\[
x_1 = a_0, \quad x_{2n} = a_n, \quad x_{2n+1} = b_n \quad \text{for } n \ge 1.
\]
Let \( F \colon s \to \mathbb{R} \) be an \( \mathcal{F}_s \)-functional. We define the corresponding functional \( G \colon C_{2\pi}^\infty(\mathbb{R}) \to \mathbb{R} \) by \( G(f) \coloneqq F(L(f)) \). By substituting the Fourier coefficients into the definition of \( F \), we obtain 
\[
G(f) = \frac{1}{2} \left( a_1 a_0^2 + \sum_{n=1}^\infty \left( a_{2n} a_n^2 + a_{2n+1} b_n^2 \right) \right) + \left( f_1(a_0) + \sum_{n=1}^\infty \left( f_{2n}(a_n) + f_{2n+1}(b_n) \right) \right).
\]
Since \(F\) satisfies the PS-condition and \( L \) is a linear homeomorphism, it follows from Proposition~\ref{prop:ps_invariance} that \(G\) also satisfies the PS-condition on \( C_{2\pi}^\infty(\mathbb{R}) \).

\paragraph{The Schwartz Space \( \mathcal{S}(\mathbb{R}) \).} Let \( \mathcal{S}(\mathbb{R}) \) be the space of rapidly decreasing functions, i.e.,
\[
\mathcal{S}(\mathbb{R}) \coloneqq \left\{ f \in C^{\infty}(\mathbb{R}) : \|f\|_{k}^{2} \coloneqq \sum_{\alpha + \beta \leq k} \int_{\mathbb{R}} |x|^{2\alpha} |f^{(\beta)}(x)|^{2} dx < \infty \text{ for all } k \in \mathbb{N} \right\}.
\]
Now we study pullback of an \( \mathcal{F}_s \)-functional to the Schwartz space \( \mathcal{S}(\mathbb{R}) \) via Hermite Expansion. 

Let \( H \colon \mathcal{S}(\mathbb{R}) \to s \) be the Hermite transform defined by
\[
H(f) \coloneqq \big( \langle f, H_{k-1} \rangle \big)_{k \in \mathbb{N}},
\]
where \( (H_n)_{n \in \mathbb{N}_0} \) are the Hermite functions forming an orthonormal basis for \( L^2(\mathbb{R}) \), and each \( H_n \in \mathcal{S}(\mathbb{R}) \). This means that \( x_k = \langle f, H_{k-1} \rangle \) for \( k \in \mathbb{N} \).  The Hermite transform \( H \) is a linear homeomorphism onto \( s \) (Example 29.5(2), \cite{vo}).

 Let \( F \colon s \to \mathbb{R} \) be an \( \mathcal{F}_s \)-functional.
 We define the functional \( G \colon \mathcal{S}(\mathbb{R}) \to \mathbb{R} \) by \( G(f) \coloneqq F(H(f)) \).
Thus, 
\[
G(f) = \frac{1}{2} \sum_{k=1}^\infty a_k \big\langle f, H_{k-1} \big\rangle^2 + \sum_{k=1}^\infty f_k\big( \langle f, H_{k-1} \rangle \big).
\]
 Since \(F\) satisfies the PS-condition and the map 
\( H  \) is a linear homeomorphism (and thus a Keller \(C_c^\infty\)-mapping), by Proposition~\ref{prop:ps_invariance}, \(G\) also satisfies the PS-condition on \( \mathcal{S}(\mathbb{R}) \).
\paragraph{Smooth Functions with Compact Support \( \mathscr{D}[a,b] \).}
Let \( \mathscr{D}[a, b] \) be the space of all smooth functions on \( \mathbb{R} \) whose support is compactly contained within the interval \( [a, b] \). That is,
\[
\mathscr{D}[a, b] \coloneqq \left\{ \phi \in C^{\infty}(\mathbb{R}) : \operatorname{supp}(\phi) \subset [a, b] \right\}.
\]
The space \( \mathscr{D}[a, b] \) is isomorphic to \( s \) (Example 29.5(3), \cite{vo}).
The overall linear homeomorphism \( L^{-1} \colon \mathscr{D}[a, b] \to s \) is a composition of the following known isomorphisms:

\begin{enumerate}
	\item {From \( \mathscr{D}[a, b] \) to \( \mathscr{D}\left[-\tfrac{\pi}{2}, \tfrac{\pi}{2}\right] \)}:
	The map \( \Psi^{-1} \colon \mathscr{D}[a, b] \to \mathscr{D}\left[-\tfrac{\pi}{2}, \tfrac{\pi}{2}\right] \) is defined by
	\[
	(\Psi^{-1}(\phi))(\xi) = \phi\left( a + \frac{b-a}{\pi}(\xi + \tfrac{\pi}{2}) \right), \quad \text{for } \phi \in \mathscr{D}[a, b] \text{ and } \xi \in \left] -\tfrac{\pi}{2}, \tfrac{\pi}{2} \right[.
	\]
	
	\item {From \( \mathscr{D}\left[-\tfrac{\pi}{2}, \tfrac{\pi}{2}\right] \) to \( \mathcal{S}(\mathbb{R}) \)}:
	The map \( \Phi^{-1} \colon \mathscr{D}\left[-\tfrac{\pi}{2}, \tfrac{\pi}{2}\right] \to \mathcal{S}(\mathbb{R}) \) is defined by
	\[
	(\Phi^{-1}(g))(\eta) = g(\arctan(\eta)), \quad \text{for } g \in \mathscr{D}\left[-\tfrac{\pi}{2}, \tfrac{\pi}{2}\right] \text{ and } \eta \in \mathbb{R}.
	\]
	\item {From \( \mathcal{S}(\mathbb{R}) \) to \( s \)}:
	This is the Hermite transform \( H \colon \mathcal{S}(\mathbb{R}) \to s \), defined by
	\[
	H(f) = \big( \langle f, H_{k-1} \rangle \big)_{k \in \mathbb{N}}, \quad \text{for } f \in \mathcal{S}(\mathbb{R}).
	\]
\end{enumerate}
Let \( \phi \in \mathscr{D}[a, b] \). The corresponding sequence \( x = (x_k)_{k \in \mathbb{N}} \in s \) that serves as the argument for \( F \) (i.e., \( x = L^{-1}(\phi) \)) is obtained by composing the following maps:

\begin{enumerate}
	\item Apply \( \Psi^{-1} \) to \( \phi \):
	Let \( g(\xi) = (\Psi^{-1}(\phi))(\xi) = \phi\left( a + \frac{b-a}{\pi}(\xi + \tfrac{\pi}{2}) \right) \).
	
	\item Apply \( \Phi^{-1} \) to \( g \):
	Let \( f(\eta) = (\Phi^{-1}(g))(\eta) = g(\arctan(\eta)) \).
	Substituting the expression for \( g(\xi) \), we get
	\[
	f(\eta) = \phi\left( a + \frac{b-a}{\pi}(\arctan(\eta) + \tfrac{\pi}{2}) \right).
	\]
	
	\item Apply \( H \) to \( f \):
	The components \( x_k \) of \( x = H(f) \) are given by the following inner product for \( k \in \mathbb{N} \):
	\[
	x_k = \langle f, H_{k-1} \rangle = \int_{\mathbb{R}} f(\eta) H_{k-1}(\eta) \, d\eta.
	\]
	Substituting the expression for \( f(\eta) \), we obtain 
	\[
	x_k (\phi)= \int_{\mathbb{R}} \phi\left( a + \frac{b-a}{\pi}(\arctan(\eta) + \tfrac{\pi}{2}) \right) H_{k-1}(\eta) \, d\eta.
	\]
\end{enumerate}
Let \( F \colon s \to \mathbb{R} \) be an \( \mathcal{F}_s \)-functional.
The functional  \( G(\phi) \coloneqq F(L^{-1}(\phi))\) is defined by
\begin{align*}
	G(\phi) = \frac{1}{2} \sum_{k=1}^\infty a_k \left( x_k(\phi) \right)^2
	+ \sum_{k=1}^\infty f_k\left( x_k(\phi) \right).
\end{align*}
 Since \(F\) satisfies the PS-condition and the composite map \( L^{-1} = H \circ \Phi^{-1} \circ \Psi^{-1} \) is a linear homeomorphism (and thus a Keller \(C_c^\infty\)-mapping), by Proposition~\ref{prop:ps_invariance}, \(G\) also satisfies the PS-condition on \( \mathscr{D}[a, b] \).
\paragraph{Smooth Functions on a Compact Interval \( C^\infty[a,b] \).}
Let \( C^\infty[a, b] \) be the space of all smooth functions on \( [a, b] \). The space \( C^\infty[a, b] \) is isomorphic to \( s \) (Example 29.5(4), \cite{vo}).
By affine scaling, we can restrict our focus to \( C^\infty[-1, 1] \). The isomorphism \( M \colon s \to C^\infty[-1, 1] \) is directly provided, mapping a sequence \( x = (x_k)_{k \in \mathbb{N}} \in s \) to a function \( f(y) \in C^\infty[-1, 1] \) using Chebyshev polynomials \( (T_n(y))_{n \in \mathbb{N}_0} \).
To align indices such that \( T_{k-1}(y) \) corresponds to \( x_k \), the map is given by
\[
M(x)(y) = \sum_{k=1}^\infty \sqrt{2\pi}\, x_k T_{k-1}(y).
\]
Let \( F \colon s \to \mathbb{R} \) be an \( \mathcal{F}_s \)-functional.
We aim to define the  functional \( G \colon C^\infty[a, b] \to \mathbb{R} \) using the established isomorphism from \( C^\infty[a, b] \) to \( s \).
To define \( G(\phi) \) for \( \phi \in C^\infty[a,b] \), we need the inverse map \( L^{-1} \colon C^\infty[a,b] \to s \). This map is a composition of the affine scaling and the inverse of \( M \).
Chebyshev polynomials form a Schauder basis for \(C^\infty[-1, 1] \). Thus,
any function \( f \in C^\infty[-1, 1] \) can be uniquely expressed as a Chebyshev series \( f(y) = \sum_{n=0}^\infty c_n(f) T_n(y) \), where the coefficients \( c_n(f) \) are given by the following orthogonality relations:
\[
c_0(f) = \frac{1}{\pi} \int_{-1}^1 f(y) \frac{dy}{\sqrt{1-y^2}}, \quad c_n(f) = \frac{2}{\pi} \int_{-1}^1 f(y) T_n(y) \frac{dy}{\sqrt{1-y^2}} \text{ for } n \ge 1.
\]
Comparing \( f(y) = \sum_{n=0}^\infty c_n(f) T_n(y) \) with the definition of \( M(x)(y) \), and aligning indices such that \( T_n(y) \) corresponds to \( x_{n+1} \), we establish the components of the sequence \( x = M^{-1}(f) \in s \) as follows
\[
x_k = \frac{1}{\sqrt{2\pi}} c_{k-1}(f) \quad \text{for } k \in \mathbb{N}.
\]
Now, to obtain the sequence for a function \( \phi \in C^\infty[a, b] \), we first apply the affine transformation \( \Lambda_{a,b} \colon C^\infty[a,b] \to C^\infty[-1,1] \) defined by
\[
f(y) = (\Lambda_{a,b}(\phi))(y) = \phi\left( a + \frac{b-a}{2}(y+1) \right).
\]
Then, we apply \( M^{-1} \) to this transformed function \( f \).
For a function \( \phi \in C^\infty[a, b] \), the corresponding sequence \( x = (x_k(\phi))_{k \in \mathbb{N}} \in s \) that serves as the argument for \( F \) has the following components:
\[
x_1(\phi) = \frac{1}{\pi\sqrt{2\pi}} \int_{-1}^1 \phi\left( a + \frac{b-a}{2}(y+1) \right) \frac{dy}{\sqrt{1-y^2}},
\]
and for \( k \ge 2 \), we have
\[
x_k(\phi) = \frac{2}{\pi\sqrt{2\pi}} \int_{-1}^1 \phi\left( a + \frac{b-a}{2}(y+1) \right) T_{k-1}(y) \frac{dy}{\sqrt{1-y^2}}.
\]
The functional \( G \colon C^\infty[a, b] \to \mathbb{R} \) is defined as \( G(\phi) \coloneqq F(x) \), where \( x = (x_k(\phi))_{k \in \mathbb{N}} \) is the sequence whose components are derived as above.
Substituting these expressions for \( x_k(\phi) \) into the definition of \( F(x) \), we obtain
\begin{align*}
	G(\phi) = \frac{1}{2} \sum_{k=1}^\infty a_k \left( x_k(\phi) \right)^2 + \sum_{k=1}^\infty f_k\left( x_k(\phi) \right).
\end{align*}
 Since \(F\) satisfies the PS-condition and the composite map \( L^{-1} \) is a linear homeomorphism (and thus a Keller \(C_c^\infty\)-mapping), by Proposition~\ref{prop:ps_invariance}, \(G\) also satisfies the PS-condition on \( C^\infty[a, b] \).
\section{Nonlinear Operator Problems via \(\mathcal{F}_s\)-Functionals}
In this section, we study nonlinear operator problems by analyzing their equivalent formulation in terms of their spectral coefficients, which naturally leads to functionals defined on \( s \).

	The problems presented in this section are formulated in full generality. However, the specific functions and sequences are chosen to satisfy Conditions~\ref{cond:an} and~\ref{cond:fn}. It should be mentioned that the applicability of the framework is not limited to these particular choices. Any other functions and sequences that satisfy Conditions~\ref{cond:an} and~\ref{cond:fn} can be applied.

\begin{remark}
	Our approach involves transforming  infinite-dimensional nonlinear operator equations into an equivalent infinite system of nonlinear algebraic equations in the sequence space \(s\). While this spectral decomposition diagonalizes the linear part and decouples the system into independent equations for each coefficient, these infinite nonlinear algebraic equations are generally not amenable to closed-form solutions.  By defining and minimizing an associated \(\mathcal{F}_s\)-functional, this method provides a  way to establish the existence, uniqueness,  and the regularity of solutions to these otherwise intractable systems. 
\end{remark}
\begin{problem}\label{p:1}
	Consider the Hilbert space \(L^2(0, \pi)\) with the orthonormal basis of sine functions
	\(
	\{ \phi_n(x) = \sqrt{\tfrac{2}{\pi}} \sin(nx) \}_{n \in \mathbb{N}}
	\).
	Let \(u(x) = \sum_{n=1}^\infty x_n \phi_n(x)\) and \(f(x) = \sum_{n=1}^\infty c_n \phi_n(x)\) be their Fourier sine series expansions. We assume that \((c_n) \in s\). For example, we choose \(c_n = \frac{1}{(n+1)!}\).
	
	Consider the problem of finding a function \(u(x)\) that satisfies the following nonlinear  operator  equation:
	\begin{equation}\label{eq:final}
		u(x) + \mathcal{K}(u(x)) + \mathcal{N}(u(x)) = f(x) \quad \text{in } L^2(0,\pi),
	\end{equation}
	where the operator \( \mathcal{K} \) is a linear, self-adjoint, compact operator defined by its action on the basis functions:
	\[
	\mathcal{K}(\phi_n(x)) = \lambda_n \phi_n(x), \quad \text{with eigenvalues } \lambda_n = \frac{1}{n^2+1}.
	\]
	The operator \( \mathcal{N} \) is defined by
	\[
	\mathcal{N}(u(x)) = \sum_{n=1}^\infty \mu_n g(x_n) \phi_n(x),
	\]
	with \( \mu_n >0 \) and \( (\mu_n) \in s \). We choose \( \mu_n = \frac{1}{n!} \) and \(g=\tanh\).
	
	Substituting the Fourier series for \(u(x)\) and \(f(x)\) into the integral equation and using the orthonormality of the basis functions, the problem transforms into the following infinite system of nonlinear equations for the coefficients \(x = (x_n)\):
	\[
	\left(1 + \frac{1}{n^2+1}\right) x_n + \frac{1}{n!} \tanh(x_n) = c_n, \quad n = 1, 2, \ldots.
	\]
	To solve the problem, we will define a functional whose critical points are the solutions.
	
	We define the functional \( F \colon s \to \mathbb{R} \) corresponding to this system by
	\[
	F(x) = \sum_{n=1}^\infty \left( \frac{1}{2} \left(1 + \frac{1}{n^2+1}\right) x_n^2 + \frac{1}{n!} \log(\cosh(x_n)) - c_n x_n \right).
	\]
	It is easy to check that critical points of \( F \) are solutions of the original nonlinear system.
	Now, we verify that \(F\) is an \( \mathcal{F}_s \)-functional. We have \( a_n = 1 + \frac{1}{n^2+1} \) and \( f_n(t) = \mu_n \log(\cosh(t)) - c_n t \).
	Each \( f_n(t) \) is \( C^\infty(\mathbb{R}) \) and strictly convex.
	Define
	\[
	\beta_n \coloneqq \mu_n + |c_n| = \frac{1}{n!} + \frac{1}{(n+1)!}> 0.
	\]
	The sequence \( (\beta_n) \in s \). We use the inequality
	\(
	0 \le \log(\cosh(t)) \le |t|\) for all \(t \in \mathbb{R}\).
	Therefore,
	\[
	\begin{aligned}
		|f_n(t)| &= |\mu_n \log(\cosh(t)) - c_n t| \\
		&\le \mu_n |t| + |c_n||t| = \beta_n |t| \le \beta_n (1 + t^2).
	\end{aligned}
	\]
	Since \( f_n(t) \) is continuous and strictly convex, and \( f_n(t) \to \infty \) as \( t \to \pm\infty \), it attains a global minimum. This minimum occurs at \(t_n\) such that \(f_n'(t_n) =0\), which gives
	\[
	\tanh(t_n) = \frac{c_n}{\mu_n} = \frac{1/(n+1)!}{1/n!} = \frac{1}{n+1}.
	\]
	The minimum value \(f_n(t_n)\) is bounded below by \(-\gamma_n\), where \( \gamma_n \approx \frac{1}{2n!(n+1)^2} \). Since the series \( \sum \gamma_n \) converges, the lower bound condition is satisfied. Thus, \(F\) is an $\mathcal{F}_s$-functional. By Corollary~\ref{cor:global_minimum}, it attains a unique global minimum $x^* = (x_n^*) \in s$. The coefficients $x_n^*$ are the unique solution to the algebraic system. Since $c_n \ne 0$, the solution $x^*$ is non-trivial.
	
	The corresponding unique solution to the original nonlinear integral equation is given by
	\(
	u^*(x) = \sum_{n=1}^\infty x_n^* \phi_n(x)
	\).
	We can now check the regularity of this solution. Since $(x_n^*) \in s$, for any integer $k \ge 0$, there exists a constant $C_k > 0$ such that
	\[
	|x_n^*| \le C_k n^{-k} \quad \text{for all } n \in \mathbb{N}.
	\]
	The $m$-th derivative of the basis functions is bounded by \( |\phi_n^{(m)}(x)| \le \sqrt{\frac{2}{\pi}} n^m \). Consider the term-wise differentiated series for \(u^{*(m)}(x)\):
	\[
	|x_n^* \phi_n^{(m)}(x)| \le |x_n^*| |\phi_n^{(m)}(x)| \le (C_k n^{-k})  \left(\sqrt{\frac{2}{\pi}} n^m\right).
	\]
	By choosing \(k = m+2\), we can find a constant \(C_{m+2}\) such that
	\[
	|x_n^* \phi_n^{(m)}(x)| \le C_{m+2} n^{-(m+2)}  \sqrt{\frac{2}{\pi}} n^m = \left(C_{m+2}\sqrt{\frac{2}{\pi}}\right) n^{-2}.
	\]
	Since the series $\sum n^{-2}$ converges, the series for $u^{*(m)}(x)$ converges absolutely and uniformly on $[0,\pi]$ for every $m \ge 0$ by the Weierstrass M-test. Therefore, the unique solution \(u^*(x)\) is a smooth function, i.e., \(u^* \in C^\infty[0,\pi]\).
\end{problem}

	As another application of our framework, we now solve a semilinear elliptic PDE. For simplicity and concreteness, we consider the one-dimensional case on $\Omega = (0, \pi)$ with Dirichlet boundary conditions.
\begin{problem}
	The Laplacian \( -\Delta \) on \(L^2(0,\pi)\) with Dirichlet boundary conditions has eigenvalues \(\lambda_n = n^2\) and a corresponding orthonormal basis of eigenfunctions \( \phi_n(x)= \sqrt{\frac{2}{\pi}} \sin(nx) \).
	Let \( f(x) = \sum_{n=1}^\infty c_n \phi_n(x) \), where \((c_n) \in s\).
	
	We aim to find the solution \(u(x)= \sum_{n=1}^\infty x_n \phi_n(x) \) to the following semilinear elliptic PDE:
	\begin{equation}\label{eq:pde}
		-\Delta u + g(u) = f(x) \quad \text{in } L^2(0, \pi)
	\end{equation}
	subject to \(u(0)=u(\pi)=0\). Here, the nonlinear term \(g(u)\) is assumed to act diagonally on the spectral coefficients (i.e., the \(n\)-th coefficient of \(g(u)\) is simply \(g(x_n)\)). We choose the specific nonlinearity \(g(t) = \tanh(t)\).
	
	By expanding \(u(x)\) and \(f(x)\) in the eigenbasis, the PDE transforms into an infinite system of algebraic equations for the coefficients \(x_n\):
	\[
	x_n + \frac{1}{n^2} \tanh(x_n) = \frac{1}{n^2} c_n \quad \text{for each } n=1, 2, \ldots.
	\]
	We define the functional \(F \colon s \to \mathbb{R}\) whose critical points solve this system:
	\[
	F(x) = \sum_{n=1}^\infty \left( \frac{1}{2} x_n^2 + \frac{1}{n^2} \ln(\cosh(x_n)) - \frac{1}{n^2} c_n x_n \right).
	\]
	To prove existence and uniqueness, we verify that this is an \(\mathcal{F}_s\)-functional.
	We set \(a_n = 1\) and \( \tilde{f}_n(t) = \frac{1}{n^2} (\ln(\cosh(t)) - c_n t) \). The coefficient sequence \( (a_n) \) clearly satisfies the required conditions. The function \(\tilde{f}_n(t)\) is smooth and strictly convex since \(\tilde{f}_n''(t) = \frac{1}{n^2} \text{sech}^2(t) > 0\).
	
	For the quadratic growth condition, let \( \beta_n \coloneqq \frac{1+|c_n|}{n^2} \). Since \( (c_n) \in s \), it is a bounded sequence, and thus \( (\beta_n) \in s \). Using the inequality \(0 \le \ln(\cosh(t)) \le |t|\), we can bound the function itself:
	\[
	|\tilde{f}_n(t)| \le \frac{1}{n^2} \left( |\ln(\cosh(t))| + |c_n t| \right) \le \frac{1}{n^2} \left( |t| + |c_n||t| \right) = \beta_n |t| \le \beta_n(1+t^2).
	\]
	This establishes the required growth condition.
	
	For the lower bound, \(\tilde{f}_n(t)\) attains a unique global minimum where \(\tilde{f}_n'(t) = 0\), which yields \( \tanh(t) = c_n \). Since \( (c_n) \in s \), we have \(c_n \to 0\), so for large \(n\), \(|c_n|<1\), guaranteeing a unique minimizer \(t_n = \operatorname{arctanh}(c_n)\).
	The minimum value is \( \tilde{f}_n(t_n) = \frac{1}{n^2} (-\frac{1}{2}\ln(1-c_n^2) - c_n \operatorname{arctanh}(c_n)) \).
	Let \(M_n = -\frac{1}{2}\ln(1-c_n^2) - c_n \operatorname{arctanh}(c_n)\). For large \(n\), as \( c_n \to 0 \), a Taylor expansion implies \( M_n \approx -\frac{1}{2}c_n^2 \).
	We can define \( \gamma_n \coloneqq - \frac{1}{n^2} M_n \ge 0 \). For large \(n\), we have \( \gamma_n \approx \frac{c_n^2}{2n^2} \). Since \( (c_n) \in s \), we can find a constant \(C\) such that \(|c_n| \le C n^{-2}\), which implies \( \gamma_n \) is bounded by terms comparable to \( n^{-6} \). Therefore, the series \( \sum_{n=1}^\infty \gamma_n \) converges absolutely.
	
	Since all conditions are met, \(F\) is an \(\mathcal{F}_s\)-functional and has a unique global minimum \(x^* = (x_n^*)\) in \(s\). The corresponding function \(u^*(x) = \sum_{n=1}^\infty x_n^* \phi_n(x)\) is the unique solution to the algebraic system. Because \(x^* \in s\), a standard argument as in Problem~\ref{p:1} using the Weierstrass M-test shows that the series for \(u^*\) and all its term-wise derivatives converge uniformly, which proves that  \(u^*\) is a smooth solution (\(u^* \in C^\infty[0,\pi]\)) to the original PDE.
\end{problem}
\begin{problem}
	Let \( H_n(t) \) be the Hermite functions, which form an orthonormal basis of \( L^2(\mathbb{R}) \) and are a basis for the Schwartz space \( \mathcal{S}(\mathbb{R}) \). We consider an unknown function $f(t) \in \mathcal{S}(\mathbb{R})$ with its Hermite series expansion $f(t) = \sum_{n=1}^\infty x_n H_{n-1}(t)$, where the coefficients $x_n = \langle f, H_{n-1} \rangle$ form a sequence in $s$.
	
	Consider the following nonlinear spectral problem:
	\[
	\sum_{n=1}^\infty \left( a_n \langle f, H_{n-1} \rangle + \nu_n h(\langle f, H_{n-1} \rangle) - c_n \right) H_{n-1}(t) = 0,
	\]
	where $(c_n)\in s$, $(\nu_n) \in s$ with $\nu_n > 0$, and $(a_n)$ is a sequence of positive real numbers such that $\alpha \le a_n \le M$. The function \(h\) can be any function such that the conditions for an $\mathcal{F}_s$-functional are met. For instance, we choose \(h=\arctan\). We assume the data sequences are such that their ratio is uniformly bounded, i.e., there exists a constant $M_{\max} < \pi/2$ with $|c_n/\nu_n| \le M_{\max}$ for all $n$. This spectral problem is equivalent to the infinite algebraic system:
	\[ a_n x_n + \nu_n \arctan(x_n) - c_n = 0, \quad \text{for } n=1,2,\dots \]
	To prove the existence and uniqueness of a solution, we consider the functional $F \colon s \to \mathbb{R}$ defined by
	\[
	F(x) = \frac{1}{2} \sum_{n=1}^\infty a_n x_n^2 + \sum_{n=1}^\infty \left( \nu_n \left(x_n \arctan(x_n) - \frac{1}{2}\log(1 + x_n^2)\right) - c_n x_n \right).
	\]
	The critical points of $F(x)$ are precisely the solutions to the algebraic system. We now verify that \(F\) is an $\mathcal{F}_s$-Functional. The sequence $(a_n)$ satisfies its condition by definition. Let 
	\[
	f_n(t) = \nu_n (t \arctan(t) - \frac{1}{2}\log(1 + t^2)) - c_n t.
	\]
 Each $f_n(t)$ is smooth and strictly convex since $f_n''(t) = \nu_n/(1+t^2)>0$. For the quadratic growth condition, we use the inequality $0 \le t \arctan(t) - \frac{1}{2}\log(1 + t^2) \le \frac{1}{2} t^2$. Thus,
	\begin{align*}
		|f_n(t)| \le \nu_n \left|t \arctan(t) - \frac{1}{2}\log(1 + t^2)\right| + |c_n t|
		\le \frac{\nu_n}{2} t^2 + |c_n||t|.
	\end{align*}
	Let $\beta_n \coloneqq \nu_n + |c_n|$. Since $(\nu_n), (c_n) \in s$, it follows that $(\beta_n) \in s$ and hence
	\[
	|f_n(t)| \le \nu_n (1+t^2) + |c_n|(1+t^2) = (\nu_n+|c_n|)(1+t^2) = \beta_n (1+t^2).
	\]
	
	The function $f_n(t)$ has a unique global minimum at $t_n^*$ where $\arctan(t_n^*) = c_n/\nu_n$. Our assumption $|c_n/\nu_n| \le M_{\max} < \pi/2$ guarantees that a unique solution $t_n^*$ exists for each $n$.
	Let $\gamma_n \coloneqq -f_n(t_n^*) = \frac{\nu_n}{2}\log(1 + (t_n^*)^2) \ge 0$.
	From the assumption, we have 
	\[
	|t_n^*| = |\tan(c_n/\nu_n)| \le \tan(M_{\max}) \mathrel{=:} K.
	\]
   Thus, the sequence of minimizers $(t_n^*)$ is uniformly bounded. Consequently,
	\[
	\gamma_n \le \frac{\nu_n}{2}\log(1+K^2).
	\]
	Since $(\nu_n) \in s$, it is in $\ell^1$, so $\sum |\nu_n| < \infty$. Therefore, 
	\[ \sum_{n=1}^\infty \gamma_n \le \frac{\log(1+K^2)}{2} \sum_{n=1}^\infty \nu_n < \infty. \]
	Thus, \(F\) is an $\mathcal{F}_s$-functional. By Corollary~\ref{cor:global_minimum}, it attains a unique global minimum $x^* = (x_n^*) \in s$. The unique solution to the nonlinear spectral problem is therefore given by the convergent Hermite series
	\[
	f^*(t) = \sum_{n=1}^\infty x_n^* H_{n-1}(t),
	\]
	and it follows from the isomorphism that $f^*(t)$ is a smooth and rapidly decreasing function, i.e., $f^*(t) \in \mathcal{S}(\mathbb{R})$.
\end{problem}

		\bigskip
	\noindent
	\textbf{Address:} Algebra and Topology Department,   Institute of Mathematics of National Academy of Sciences of Ukraine, 
	Tereshchenkivska st. 3,  01024, Kyiv, Ukraine \\
	\textbf{Email:} kaveh@imath.kiev.ua
\end{document}